\documentclass[a4paper,11pt,reqno]{amsart}

\usepackage{amsmath}
\usepackage{amsthm}
\usepackage{amssymb}
\usepackage{latexsym}
\usepackage{graphicx,color}
\usepackage{booktabs}

\numberwithin{equation}{section}

\newtheorem{thm}{Theorem}[section]
\newtheorem{prop}[thm]{Proposition}
\newtheorem{lem}[thm]{Lemma}
\newtheorem{cor}[thm]{Corollary}

\theoremstyle{definition}
\newtheorem{defn}[thm]{Definition}

\theoremstyle{remark}
\newtheorem{rem}[thm]{Remark}

\allowdisplaybreaks[4]

\newcommand{\e}{\varepsilon}

\newcommand{\la}{\lambda}

\newcommand{\sgm}{\sigma}
\renewcommand{\th}{\theta}
\newcommand{\om}{\omega}

\newcommand{\p}{\partial}
\newcommand{\I}{\infty}
\newcommand{\Sc}[1]{\mathcal{#1}}
\newcommand{\F}{\Sc{F}}

\newcommand{\Bo}[1]{\mathbb{#1}}
\newcommand{\R}{\Bo{R}}
\newcommand{\T}{\Bo{T}}

\newcommand{\lec}{\lesssim}
\newcommand{\gec}{\gtrsim}

\newcommand{\bbar}{\overline}
\newcommand{\ti}{\widetilde}

\newcommand{\shugo}[1]{\{ #1\}}
\newcommand{\Shugo}[2]{\big\{ \, #1 : #2 \, \big\}}

\newcommand{\LR}[1]{{\langle #1 \rangle }}
\newcommand{\chf}[1]{\textbf{1}_{#1}}

\newcommand{\dyadic}{2^{\Bo{N}_0}}

\newcommand{\norm}[2]{\big\| #1 \big\| _{#2}}
\newcommand{\tnorm}[2]{\| #1 \| _{#2}}

\newcommand{\eq}[2]{\begin{equation} \label{#1} \begin{split} #2 \end{split} \end{equation}}
\newcommand{\eqq}[1]{\begin{equation*} \begin{split} #1 \end{split} \end{equation*}}
\newcommand{\eqs}[1]{\begin{gather*} #1 \end{gather*}}
\newcommand{\mat}[1]{\begin{smallmatrix} #1 \end{smallmatrix}}

\newcommand{\hx}{\hspace{10pt}}

\newcommand{\hz}{\!\!\!}

\title[Unconditional LWP for periodic NLS]{Unconditional local well-posedness for periodic NLS}
\author[N. Kishimoto]{Nobu Kishimoto}
\address{Research Institute for Mathematical Sciences, Kyoto University, Kyoto 606-8502, Japan}
\email{nobu@kurims.kyoto-u.ac.jp}

\begin{document}

\begin{abstract}
The nonlinear Schr\"odinger equations with nonlinearities $|u|^{2k}u$ on the $d$-dimensional torus are considered for arbitrary positive integers $k$ and $d$.
The solution of the Cauchy problem is shown to be unique in the class $C_tH^s_x$ for a certain range of scale-subcritical regularities $s$, which is almost optimal in the case $d\ge 4$ or $k\ge 2$.
The proof is based on various multilinear estimates and the infinite normal form reduction argument.
\end{abstract}

\maketitle

\section{Introduction}

We consider the Cauchy problem for nonlinear Schr\"odinger equations with periodic boundary condition:
\begin{equation}\label{NLS}
\left\{
\begin{array}{@{\,}r@{\;}l@{\qquad}r@{\;}l}
i\p _tu+\Delta u&=\la |u|^{2k}u,&(t,x)&\in [0,T]\times \T ^d,\quad d,k\in \Bo{N},\quad \la \in \Bo{C},\\
u(0,x)&=u_0(x),&x&\in \T ^d,
\end{array}
\right. 
\end{equation}
where $\T ^d:=\R ^d/2\pi \Bo{Z}^d$ is the $d$-dimensional torus.
The purpose of this article is to show unconditional (local) well-posedness of \eqref{NLS} in low-regularity Sobolev spaces $H^s(\T^d)$ for general dimensions $d$ and degrees of nonlinearity $2k+1$ by means of an abstract theory given in \cite{K-all} based on the normal form reduction technique.
Here, ``unconditional'' means that uniqueness of the solution in the sense of distribution holds in the entire space $C([0,T];H^s)$.
We distinguish it from ``conditional'' well-posedness, for which uniqueness is ensured in a subset of $C([0,T];H^s)$ or under additional assumptions, depending on how the solution is constructed.
For instance, a standard iteration argument with Sobolev inequalities shows that \eqref{NLS} is unconditionally locally well-posed in $H^s(\T ^d)$ for $s>\frac{d}{2}$, while the Fourier restriction norm method (or Bourgain's method, see \cite{B93-1}) may yield conditional local well-posedness for lower regularities, in which case uniqueness of solutions would be shown only in Bourgain spaces. 

The (conditional) local well-posedness of \eqref{NLS} on the torus, along with underlying periodic Strichartz estimates of the form
\eqq{\norm{P_{\le N}e^{it\Delta}\phi}{L^p_{t,x}(I\times \T ^d)}\lec _{d,p,I}N^{\theta (d,p)}\tnorm{\phi}{L^2(\T ^d)},}
has been quite extensively studied since the pioneering work of Bourgain~\cite{B93-1}.
In \cite{B93-1} the Cauchy problem \eqref{NLS} on standard square (or rational) tori was treated and its local and global well-posedness in $H^s(\T ^d)$ was established already for a wide range of $d$, $k$ and scale-subcritical regularities $s$.
Accordingly, a major interest nowadays has been drawn by the problems at the scale-critical regularities and posed on general irrational tori; see \cite{B07,B13,BD15,CW10,D14p,D17,GOW14,HTT11,HTT12,KV16,L19,S14,W13}, for instance.
For \eqref{NLS} posed on the square torus, local well-posedness in $H^s$ is known to hold for any $d,k\in \Bo{N}$ and any subcritical/critical regularities $s\ge s_c$,
\eqq{s_c:=\frac{d}{2}-\frac{1}{k},}
with the exception of 1d cubic case $d=k=1$ (where the Cauchy problem is globally well-posed in $L^2$ but ill-posed in any Sobolev space of negative index; see \cite{B93-1,CCT03p,M09,GO18}) and $L^2$-critical cases $(d,k)=(1,2),(2,1)$ (where well-posedness in the critical space $H^{s_c}=L^2$ is open; see \cite{K14} for a partial result).

Concerning unconditional well-posedness, there are two natural thresholds: 
One is $s\ge s_c$ coming from the scaling, and the other is $s\ge s_e$,
\eqq{s_e:=\frac{d(2k-1)}{2(2k+1)},}
which is needed for the embedding $H^s\hookrightarrow L^{2k+1}$ so that the nonlinearity $|u|^{2k}u$ makes sense within the framework of distribution.
Therefore, the natural conjecture is that \eqref{NLS} is unconditionally well-posed in $H^s(\Bo{T}^d)$ for
\eq{conjecture}{s\ge \max \shugo{s_c,\,s_e}=\max \shugo{\tfrac{d}{2}-\tfrac{1}{k},\,\tfrac{d}{2}\tfrac{2k-1}{2k+1}}.}
We see that $s_c=s_e+\frac{1}{2k+1}(d-2-\frac{1}{k})$. 
In particular, $s_c\ge s_e$ if and only if $d\ge 2+\frac{1}{k}$, and
\eqq{
\begin{cases}
s_c<s_e &\text{if $d=1,2$,}\\
s_c=s_e &\text{if $d=3$ and $k=1$,}\\
s_c>s_e &\text{if $d=3$ and $k\ge 2$, or $d\ge 4$.}
\end{cases}
}

Unconditional well-posedness for nonlinear Schr\"odinger equations was investigated first by Kato~\cite{K95} and has been well studied in the non-periodic case, while much less is known in the periodic case.
Guo, Kwon, and Oh \cite{GKO13} proved unconditional uniqueness of the solution for \eqref{NLS} with $d=k=1$ in $H^s(\T )$ under natural regularity constraint $s\ge \frac{1}{6}=\max \shugo{s_c,\,s_e}$ via the technique of (Poincar\'e-Dulac) normal form reduction.
Chen and Holmer~\cite{CH19} and Herr and Sohinger~\cite{HS19} obtained uniqueness results on \eqref{NLS} with $\la =\pm 1$ from the analysis of the Gross-Pitaevskii hierarchy, for quintic (defocusing) NLS on the 3d square torus at the critical regularity $s=s_c=1$ and for cubic NLS on arbitrary (irrational) tori in dimension two and higher with regularities in a certain subcritical range, respectively.
(We will come back to these results in Remark~\ref{rem:CHandHS} later.)
Recently, the author~\cite{K-all} abstracted the methodology introduced in \cite{GKO13} and proved unconditional uniqueness for \eqref{NLS} on the 2d square torus in $H^{2/3}$ as an application of the abstract theory.

In this article, we extend the result on unconditional uniqueness for \eqref{NLS} to general higher-dimensional/higher-degree cases by applying the abstract result in \cite{K-all} again, but with more refined analysis.
The main result reads as follows:
\begin{thm}\label{thm:uniqueness}
Let $d,k\in \Bo{N}$ with $(d,k)\neq (1,1)$.
Assume that $s$ satisfies
\eq{cond:thm}{
\begin{cases}
~s>\frac{3d-2}{10}=\left\{ \begin{split} ~&\tfrac{2}{5} \\[3pt] &\tfrac{7}{10} \end{split} \right. \quad &\text{if $k=1$ and}~\left\{ \begin{split} &d=2, \\[3pt] &d=3, \end{split} \right. \\[10pt]
~\text{$s>s_c$ and $s\ge s_e$} &\text{if $k\ge 2$ or $d\ge 4$.}
\end{cases}
}
Then, \eqref{NLS} is unconditionally locally well-posed in $H^s(\Bo{T}^d)$.
\end{thm}

\begin{rem}
(i) Note that the lower bound \eqref{cond:thm} in the case $k=1$, $d=2,3$ satisfies
\eqq{\tfrac{d}{2}>\tfrac{3d-2}{10}>\tfrac{d}{6}=\max \shugo{s_c,s_e}.}
For any other cases, we have almost optimal results in view of \eqref{conjecture}.
In particular, if the endpoint is subcritical with respect to scaling (\mbox{i.e.} $s_e>s_c$), unconditional local well-posedness holds also at the endpoint, unless $d=2$ and $k=1$.
On the other hand, the abstract theory in \cite{K-all} is basically not prepared for application to the scale-critical problem.
As a result, in some cases (e.g., $d\ge 4$ and any $k$) conditional well-posedness is already known at the endpoint $s=s_c=\max \shugo{s_c,s_e}$ but unconditional uniqueness is left open.

(ii) Since we do not utilize Hamiltonian structure or conservation laws of the equation, the result holds equally for any $\la \in \Bo{C}$.
Moreover, our method can be applied to equations with non-gauge-invariant power-type nonlinearities as well.

(iii) In the proof, we only use a ``classical'' argument based on divisor counting, so that our result is restricted to the case of square (or rational) tori.
It would be of interest to investigate the problem on general irrational tori by adapting ``modern'' techniques from \cite{BD15} to the abstract theory in \cite{K-all}. 
\end{rem}

\begin{rem}\label{rem:CHandHS}
We show uniqueness among all distributional solutions in $C([0,T];H^s)$ without assuming any other property.
This should be compared with the results in \cite{CH19,HS19}, where  uniqueness was shown for the solution \emph{satisfying certain conservation laws}. 
In fact, conservation of the energy and of the $L^2$-norm played an essential role in the argument in \cite{CH19} and in \cite{HS19}, respectively.

On one hand, for the cubic NLS one can show the $L^2$ conservation law for any distributional solution in $C([0,T];H^{d/4})$.
To see this, let $\lambda =1$ for simplicity and consider the approximating sequence $u_N:=P_{\le N}u$ of such a solution $u$.
Here, $P_{\le N}u:=\F ^{-1}_n\chf{\{ |n|\le N\}}(\F _xu)(t,n)$, and $\F _x$, $\F^{-1}_n$ denote the Fourier and inverse Fourier transforms on $\T ^d$ and $\Bo{Z}^d$, respectively.
The following identity is verified by the equation for $u_N$ (which is smooth): 
\eqq{\tnorm{u_N(t)}{L^2}^2=\tnorm{u_N(0)}{L^2}^2+2\Im \int _0^t\int _{\T ^d}\Big( P_{\le N}(|u|^2u)-|u_N|^2u_N\Big) \bbar{u_N}\,dx\,dt ,\qquad t\in [0,T].}
If $u\in C([0,T];H^{d/4})\hookrightarrow C([0,T];L^4)$, the integral on the right-hand side vanishes as $N\to \I$, which shows that the $L^2$-norm of $u$ is constant in time.
Recall that in \cite{HS19} uniqueness was claimed in $H^s(\T ^d)$ for $s>\frac{7}{12}~(\ge \frac{1}{2})$ if $d=2$; $s>\frac{4}{5}~(\ge \frac{3}{4})$ if $d=3$; $s>\frac{d}{2}-1~(\ge \frac{d}{4})$ if $d\ge 4$. 
Therefore, the uniqueness result in \cite{HS19} actually yields unconditional uniqueness in the same regularity range.
In particular, Theorem~\ref{thm:uniqueness} is covered by this result in the case $k=1$, $\la \in \R$, and for $s$ within the above range.

On the other hand, a similar regularization argument would not show the energy conservation of the 3d quintic NLS for general solutions in $C([0,T];H^1)$.
In fact, we have
\eqq{\frac{1}{2}\tnorm{\nabla u_N(t)}{L^2}^2+\frac{1}{6}\tnorm{u_N(t)}{L^6}^6&=\frac{1}{2}\tnorm{\nabla u_N(0)}{L^2}^2+\frac{1}{6}\tnorm{u_N(0)}{L^6}^6\\
&\quad +\Im \int _0^t \int _{\T ^3}\nabla \Big( P_{\le N}(|u|^4u)-|u_N|^4u_N\Big) \cdot \nabla \bbar{u_N}\,dx\,dt\\
&\quad +\Im \int _0^t \int _{\T ^3}\Big( P_{\le N}(|u|^4u)-|u_N|^4u_N\Big) |u_N|^4\bbar{u_N}\,dx\,dt.}
It seems that the treatment of the first (resp.~the second) integral on the right-hand side requires the regularity at least $H^{4/3}$ (resp.~$H^{6/5}$).
Hence, it is not clear whether the uniqueness result in \cite{CH19} implies unconditional uniqueness (without assuming conservation of the energy) for the 3d quintic NLS at the critical regularity $H^1$.
It shows unconditional uniqueness in the class $C([0,T];H^{4/3})$, however.
\end{rem}

Since local well-posedness of \eqref{NLS} has been obtained in the whole subcritical range of regularities, to prove Theorem~\ref{thm:uniqueness} we only have to show unconditional uniqueness.
Following the argument in \cite{K-all}, we first move to the equation on the frequency side.
Let $u(t)\in C([0,T];H^{s_e}(\T ^d))$ be a solution (in the sense of distribution) of \eqref{NLS} and $\om (t,n):=\F _x[e^{-it\Delta}u(t)](n)$, then $\om$ satisfies
\eq{eq_v1}{\p _t\om (t,n)=c\lambda \hspace{-10pt}\sum _{\mat{n_1,n_2,\dots ,n_{2k+1}\in \Bo{Z}^d\\ n=n_1-n_2+\dots -n_{2k}+n_{2k+1}}}\hspace{-10pt}e^{it\Phi}\om (t,n_1)\bbar{\om (t,n_2)}\cdots \bbar{\om (t,n_{2k})}\om (t,n_{2k+1}),\qquad n\in \Bo{Z}^d,}
where $c$ is a constant depending on the definition of the Fourier transform and
\eqq{\Phi =\Phi (n,n_1,n_2,\dots ,n_{2k+1}):=&\,|n|^2-|n_1|^2+|n_2|^2-\dots -|n_{2k+1}|^2.}
Note that the sum in \eqref{eq_v1} is absolutely convergent for each $n$, since $\F ^{-1}_n|\om (t)|\in H^{s_e}\subset L^{2k+1}$.
In particular, $\om (\cdot ,n)\in C^1([0,T])$ for each $n\in \Bo{Z}^d$ and \eqref{eq_v1} holds in the classical sense. 

We next separate some terms from the nonlinear part.
This step was not taken in \cite[Section~3]{K-all} for the sake of simplicity, while we do in order to obtain uniqueness in lower regularities.
\begin{defn}\label{defn:A}
Let $d,k\in \Bo{N}$ with $(d,k)\neq (1,1)$.
We define the set $\Sc{A}\subset (\Bo{Z}^d)^{2k+1}$ as follows.
\begin{itemize}
\item If $d\ge 2+\frac{2}{k}$ (i.e., $d\ge 4$, or $d=3$ and $k\ge 2$), then $\Sc{A}:=\emptyset$.
\item For $k=1$ and $d=2,3$, we define
\eqq{\Sc{A}:=\Shugo{(n_1,n_2,n_3)\in (\Bo{Z}^d)^{3}}{\text{$n_2=n_1$ or $n_2=n_3$}}.}
\item For $k\ge 2$ and $d=1,2$, we first fix a linear order $\succeq$ on $\Bo{Z}^d$ such that $n_1\succeq n_2$ implies $|n_1|\ge |n_2|$. 
For instance, we may define it as
\eqq{n_1\succeq n_2\quad \Leftrightarrow \quad 
\left\{ \begin{split}
~&\text{$|n_1|>|n_2|$, or}\\[-5pt]
&\text{$|n_1|=|n_2|$ and $n_1\ge n_2$ in the lexicographic order on $\Bo{Z}^d$.}
\end{split}\right.
}
Given $(n_l)_{l=1}^{2k+1}\in (\Bo{Z}^d)^{2k+1}$ and $m\in \{ 1,2,\dots ,2k+1\}$, we write $n_{[m]}$ to denote the $m$-th largest one in the order $\succeq$ among $\shugo{n_l}_{l=1}^{2k+1}$.
Then, we define
\eqq{\Sc{A}:=\begin{cases}
\Sc{A}_1\cup \Sc{A}_2 &\text{if $k\ge 3$ or $d=1$},\\
\Sc{A}_1\cup \Sc{A}_2\cup \Sc{A}_3 &\text{if $k=2$ and $d=2$},
\end{cases}}
where
\eqq{\Sc{A}_1&:=\Shugo{(n_l)_{l=1}^{2k+1}}{\text{$n_{[1]}=n_{[2]}$}},\\
\Sc{A}_2&:=\Shugo{(n_l)_{l=1}^{2k+1}}{\text{$n_{[2]}=n_{[3]}$}},\\
\Sc{A}_3&:=\Shugo{(n_l)_{l=1}^{2k+1}}{\text{$n_{[3]}=n_{[4]}$ and $\LR{n_{[2]}}\le \LR{n_{[3]}}^{3/2}$}},\quad \LR{n}:=(1+|n|^2)^{1/2}.}
\end{itemize}
\end{defn}

The set $\Sc{A}$ consists of the frequencies which we separate from the principal nonlinear part.
Namely, we rewrite the equation \eqref{eq_v1} as
\eqq{\p _t\om (t,n)=c\la \hspace{-10pt}\sum _{\mat{n_1,n_2,\dots ,n_{2k+1}\in \Bo{Z}^d\\ n=n_1-n_2+\dots +n_{2k+1}}}\hspace{-10pt}\chf{\Sc{A}^c}e^{it\Phi}\om(t,n_1)\bbar{\om(t,n_2)}\cdots \om(t,n_{2k+1})+\Sc{R}[\om (t)](n),\quad n\in \Bo{Z}^d,}
where
\eqq{\Sc{R}[\om (t)](n):=c\la \sum _{\mat{n_1,n_2,\dots ,n_{2k+1}\in \Bo{Z}^d\\ n=n_1-n_2+\dots +n_{2k+1}}}\chf{\Sc{A}}e^{it\Phi}\om(t,n_1)\bbar{\om(t,n_2)}\cdots \om(t,n_{2k+1}).}
We use the notation of weighted sequential $L^p$-norms $\tnorm{\om}{\ell ^p_s}:=\tnorm{\LR{\cdot}^s\om (\cdot )}{\ell ^p(\Bo{Z}^n)}$ for $p\in [1,\I ]$ and $s\in \R$.
The main ingredient of the proof is to show the following multilinear estimates:
\begin{prop}\label{prop:fundamental}
Let $d,k$ be positive integers such that $(d,k)\neq (1,1)$.
The following holds.
\begin{enumerate}
\item For any $s_1>s_c~(\ge 0)$ and $s_2>d/2$, we have
\eqq{
\text{(B1)}\qquad &\sup _{\mu \in \Bo{Z}}\norm{\sum _{\mat{n_1,n_2,\dots ,n_{2k+1}\in \Bo{Z}^d\\ n=n_1-n_2+\dots +n_{2k+1}}}\chf{\{ \Phi =\mu \}}\prod _{l=1}^{2k+1}\om _l(n_l)}{\ell ^2_{s_1}(\Bo{Z}^d_n)}\lec\prod _{l=1}^{2k+1}\tnorm{\om _l}{\ell ^2_{s_1}},\\
\text{(B1)'}\qquad &\norm{\sum _{\mat{n_1,n_2,\dots ,n_{2k+1}\in \Bo{Z}^d\\ n=n_1-n_2+\dots +n_{2k+1}}}\prod _{l=1}^{2k+1}\om _l(n_l)}{\ell ^2_{s_2}(\Bo{Z}^d_n)}\lec \prod _{l=1}^{2k+1}\tnorm{\om _l}{\ell ^2_{s_2}}.
}
\item  Let $s>s_c$ satisfy the condition \eqref{cond:thm}. 
Define $s_2:=\max \{ \frac{d}{2},s\} +1$, and $s_1\in (s_c,s)$, $r\in [2,\I ]$, $\sgm \in [-s,0]$ by
\eqq{
[s_1,r,\sgm ]:=\begin{cases}
\big[ \frac{s+s_c}{2},\,2,\,-s_c\big] &\text{if $d\ge 2+\frac{2}{k}$,}\\[5pt]
\big[ \max \{ \frac{s+s_c}{2},s_e-\frac{1}{2}\e (k)\} ,\,\I ,\,0\big] &\text{if $d=1,2$ and $k\ge 2$,}\\[5pt]
\big[ \frac{1}{2}(s+\frac{3d-2}{10}),\,\frac{10}{2d-3},\,-\frac{(2d-3)(d-2)}{10}\big] &\text{if $d=2,3$ and $k=1$,}
\end{cases}
}
where $\e (k)$ is a positive constant given in Lemma~\ref{lem:trilinear2} below.
Then, it holds that
\eqq{
\text{(R)}\qquad &\norm{\sum _{\mat{n_1,n_2,\dots ,n_{2k+1}\in \Bo{Z}^d\\ n=n_1-n_2+\dots +n_{2k+1}}}\chf{\Sc{A}}\prod _{l=1}^{2k+1}\om _l(n_l)}{\ell ^2_s(\Bo{Z}^d_n)}\lec \prod _{l=1}^{2k+1}\tnorm{\om _l}{\ell ^2_s},\\
\text{(B2)}\qquad &\hz\hz\sup _{\mu \in \Bo{Z}}\norm{\hz\sum _{\mat{n_1,n_2,\dots ,n_{2k+1}\in \Bo{Z}^d\\ n=n_1-n_2+\dots +n_{2k+1}}}\hz\hz\hz\chf{\Sc{A}^c\cap \{ \Phi =\mu \}}\prod _{l=1}^{2k+1}\om _l(n_l)}{\ell ^r_{\sigma}(\Bo{Z}^d_n)}\lec \min _{1\le q\le 2k+1}\tnorm{\om _q}{\ell ^r_{\sigma}}\prod _{\mat{l=1\\ l\neq q}}^{2k+1}\tnorm{\om _l}{\ell ^2_{s_1}},\\
\text{(B2)'}\qquad &\norm{\sum _{\mat{n_1,n_2,\dots ,n_{2k+1}\in \Bo{Z}^d\\ n=n_1-n_2+\dots +n_{2k+1}}}\chf{\Sc{A}^c}\prod _{l=1}^{2k+1}\om _l(n_l)}{\ell ^r_{\sigma}(\Bo{Z}^d_n)}\lec \min _{1\le q\le 2k+1}\tnorm{\om _q}{\ell ^r_{\sigma}}\prod _{\mat{l=1\\ l\neq q}}^{2k+1}\tnorm{\om _l}{\ell ^2_{s_2}},\\
\text{(B3)}\qquad &\norm{\sum _{\mat{n_1,n_2,\dots ,n_{2k+1}\in \Bo{Z}^d\\ n=n_1-n_2+\dots +n_{2k+1}}}\chf{\Sc{A}^c}\prod _{l=1}^{2k+1}\om _l(n_l)}{\ell ^r_\sigma (\Bo{Z}^d_n)}\lec \prod _{l=1}^{2k+1}\tnorm{\om _l}{\ell ^2_s}.
}
\end{enumerate}
\end{prop}
Once Proposition~\ref{prop:fundamental} is proved, Theorem~\ref{thm:uniqueness} follows from \cite[Theorem~1.1]{K-all}. 
See \cite[Section~3]{K-all} for details.

\begin{rem}
We can also obtain existence of local-in-time weak solutions to \eqref{NLS} for any subcritical regularities $s>s_c$ (unless $d=k=1$) by combining Proposition~\ref{prop:fundamental} (i) above with Theorem~7.3 (and the argument in Remark~7.4) in \cite{K-all}.
However, this result is not so meaningful as unconditional uniqueness shown in Theorem~\ref{thm:uniqueness}, since the weak solutions constructed by \cite[Theorem~7.3]{K-all} turn out to be identical with the (distributional) solutions constructed in former works by the fixed point argument in $X^{s,b}$- or $U^2$, $V^2$-type spaces using the Strichartz estimates.
Note that \cite[Theorem~7.3]{K-all} does not in itself imply any property of the weak solutions by which the nonlinearity can make sense within the distributional framework.
\end{rem}

In Section~\ref{sec:number} we will prove several estimates on the number of lattice points satisfying certain relations.
These estimates will be used to prove Proposition~\ref{prop:fundamental} in Section~\ref{sec:proof}.


\bigskip
\section{Preliminaries}\label{sec:number}

In this section, we prepare several estimates on the number of lattice points satisfying certain relations.
These estimates are the key ingredients of the proof of the main multilinear estimates and shown by use of combinatorial tools such as the divisor bound: For any $\e >0$ there is $C>0$ such that $\# \Shugo{m\in \Bo{N}}{\text{$m$ divides $n$}}\le Cn^\e$ for any positive integer $n$.

\begin{lem}\label{lem:numbertheory}
Let $d\ge 2$.
Then, for any $\eta >0$ there exists $C>0$ such that
\begin{gather}
\label{numberA} \# \Shugo{n\in \Bo{Z}^d}{|n-n^*|^2=\mu ^*,\,n\in B_R^d}\le CR^{d-2+\eta},\\
\label{numberB} \# \Shugo{(p,q)\in \Bo{Z}^2}{(p-p^*)^2+3(q-q^*)^2=\mu ^*,\,(p,q)\in B_R^2}\le CR^{\eta},
\end{gather}
for any $n^*\in \Bo{Z}^d$, $(p^*,q^*)\in \Bo{Z}^2$, $\mu ^*\ge 0$, and any ball $B_R^d\subset \R^d$ of radius $R>1$.
\end{lem}

\begin{proof}
We only consider the case $d=2$ for \eqref{numberA}; for $d\ge 3$, it suffices to fix $d-2$ components of $n$ (which amounts to $CR^{d-2}$) and then apply the 2d bound for the remaining two components.

When $\mu ^*\lec R^6$, we recall the well-known estimate on the number of lattice points on a circle for the estimate \eqref{numberA}, while for \eqref{numberB} we refer to the argument in \cite{B93-1}, p.117.
When $\mu ^*\gg R^6$, Jarn\'ik's geometric observation \cite{J26} shows that there are at most two points; see \mbox{e.g.} Lemma~1.5 in \cite{B07}.
\end{proof}

\begin{cor}\label{cor:number}
We have the following estimates.

\noindent (i) For any $\eta >0$ there exists $C>0$ such that for any $R>1$ and $n^*,n_*,\mu ^*\in \Bo{Z}$,
\eq{est-1'+}{
\# \Shugo{(n_1,n_2,n_3)\in \Bo{Z}^3}{n_1+n_2+n_3=n^*,\,n_1^2+n_2^2+n_3^2=\mu ^*,\,& \\
|n_1-n_*|+|n_2|\le R} &\le CR^{\eta}. 
}

\noindent (ii) Let $d\ge 2$.
For any $\eta >0$ there is $C>0$ such that for any $R,R_1,R_2>1$, $n^*,n_*\in \Bo{Z}^d$, $\mu ^*\in \Bo{Z}$,
\eq{est-d+}{
&\# \Shugo{(n_1,n_2)\in (\Bo{Z}^d)^2}{n_1+n_2=n^*,\,|n_1|^2+|n_2|^2=\mu ^*,\,|n_1-n_*|\le R}\le CR^{d-2+\eta},
}
\eq{est-d'+}{
\# \Shugo{(n_1,n_2,n_3)\in (\Bo{Z}^d)^2}{~&n_1+n_2+n_3=n^*,\,|n_1|^2+|n_2|^2+|n_3|^2=\mu ^*,\\
&|n_1|\le R_1,\,|n_2|\le R_2} \le C\max \shugo{R_1,R_2}^{d-2+\eta}\min \shugo{R_1,R_2}^d.
}
\end{cor}

\begin{proof}
(i) The condition for \eqref{est-1'+} implies
\eqs{\big( (3n_1-3n^*) +2n^*\big) ^2+3(n_1-n^*+2n_2) ^2=6\mu ^*-2(n^*)^2,\\
(3n_1-3n_*)^2+(n_1-n^*+2n_2)^2\le 16R^2.}
\eqref{est-1'+} then follows from \eqref{numberB} with $p=3n_1-3n^*$ and $q=n_1-n^*+2n_2$.

(ii) For \eqref{est-d+}, the condition implies
\eqq{\big| (2n_1-2n^*)+n^*\big| ^2=2\mu ^*+|n^*|^2,\quad | 2n_1-2n^*| \le 2R,}
and the estimate follows by \eqref{numberA} with $n=2n_1-2n^*$.
\eqref{est-d'+} follows easily from \eqref{est-d+} fixing one of $n_1,n_2$ first.
\end{proof}

\begin{lem}\label{lem:number2}
We have the following estimates.

\noindent (i) For any $\eta >0$ there exists $C>0$ such that for any $R>1$ and $n^*,n_*,\mu ^*\in \Bo{Z}$,
\begin{gather}
\begin{split}
\# \Shugo{(n_1,n_2,n_3)\in \Bo{Z}^3}{n_1-n_2+n_3=n^*,\,n_1^2-n_2^2+n_3^2=\mu ^*,\,&\\
n_2\neq n_1,\,n_2\neq n_3,\,|n_1|+|n_3|\le R} &\le CR^{\eta},
\end{split} \label{est-1'-1}\\[5pt]
\begin{split}
\# \Shugo{(n_1,n_2,n_3)\in \Bo{Z}^3}{n_1-n_2+n_3=n^*,\,n_1^2-n_2^2+n_3^2=\mu ^*,\,&\\
n_2\neq n_1,\,n_2\neq n_3,\,|n_1|+|n_2|\le R} &\le CR^{\eta}.
\end{split} \label{est-1'-2}
\end{gather}

\noindent (ii) Let $d\ge 2$.
For any $\eta >0$ there is $C>0$ such that for any $R,R_1,R_2,R_3>1$, $n^*\in \Bo{Z}^d$, $\mu ^*\in \Bo{Z}$,
\begin{gather}
\label{est-d-} \# \Shugo{(n_1,n_2)\in (\Bo{Z}^d)^2}{n_1-n_2=n^*\neq 0,\,|n_1|^2-|n_2|^2=\mu ^*,\,|n_1|\le R}\le CR^{d-1},\\[5pt]
\begin{split}
\# \big\{ \,(n_1,n_2,n_3)\in (\Bo{Z}^d)^3:n_1-n_2+n_3=n^*,\,|n_1|^2-|n_2|^2+|n_3|^2=\mu ^*,\,&\\
n_2\neq n_1,\,n_2\neq n_3,\,|n_1|\le R_1,\,|n_3|\le R_3\,\big\} \le CR_1^{d-1}R_3^{d-1}&\max \{ R_1,R_3\} ^{\eta},
\end{split} \label{est-d'-1}\\[5pt]
\begin{split}
\# \big\{ \,(n_1,n_2,n_3)\in (\Bo{Z}^d)^3:n_1-n_2+n_3=n^*,\,|n_1|^2-|n_2|^2+|n_3|^2=\mu ^*,\,&\\
n_2\neq n_1,\,n_2\neq n_3,\,|n_1|\le R_1,\,|n_2|\le R_2\,\big\} \le CR_1^{d-1}R_2^{d-1}&\max \{ R_1,R_2\}^{\eta},
\end{split} \label{est-d'-2}\\[5pt]
\begin{split}
\# \Shugo{(n_1,n_2,n_3)\in (\Bo{Z}^d)^3}{~&n_1-n_2+n_3=n^*,\,|n_1|^2-|n_2|^2+|n_3|^2=\mu ^*,\\
&|n_1|\le R_1,\,|n_3|\le R_3} \le C\max \{ R_1,R_3\} ^{d}\min \{ R_1,R_3\} ^{d-2+\eta}.
\end{split} \label{est-d'-}
\end{gather}
\end{lem}

\begin{proof}
\eqref{est-1'-1}: We deduce from the condition that
\eqq{0\neq (n^*-n_1)(n^*-n_3)=\mu ^{**}:=((n^*)^2-\mu ^*)/2.}
If $|\mu ^{**}|\lec R^6$, the divisor bound implies that there are at most $O(R^\eta )$ choices for $n^*-n_1$ and $n^*-n_3$, which determine $(n_1,n_2,n_3)$.
If $|\mu ^{**}|\gg R^6$, we see that $|n^*|\sim |\mu ^{**}|^{1/2}\gg R^3\ge |n_1|,|n_3|$.
We may assume $n^*-n_1\ge n^*-n_3>0$. 
It turns out that there are at most two choices for such $(n_1,n_2,n_3)$ (\mbox{cf.} \cite{CKSTT04}, Lemma~6.1).
In fact, suppose that there are three different triplets, and let $a,b,c$ be the corresponding values for $n^*-n_1$; hence $a\sim b\sim c\sim n^* \sim (\mu ^{**})^{1/2}\gg R^3$ and they are mutually different.
Since $a,b,c$ divide $\mu^{**}$, $\mathrm{lcm}(a,b,c)$ also divides $\mu ^{**}$ and thus not greater than $\mu ^{**}$.
Moreover, since $a,b,c$ are confined to the interval $[n^*-R,n^*+R]$, we have 
\eqq{\mathrm{gcd}(a,b),\hx \mathrm{gcd}(a,c),\hx \mathrm{gcd}(b,c)\le 2R}
by the Euclidean algorithm.
Now, the identity
\eqq{\mathrm{lcm}(a,b,c)\mathrm{gcd}(a,b)\mathrm{gcd}(a,c)\mathrm{gcd}(b,c)=abc\cdot \mathrm{gcd}(a,b,c)}
from elementary number theory shows that
\eqq{\mu ^{**}(2R)^3\ge abc \sim (\mu ^{**})^{3/2},}
which contradicts the assumption $|\mu ^{**}|\gg R^6$.

\eqref{est-1'-2}: Similarly, we have
\eqq{0\neq (n^*-n_1)(n_1-n_2)=\mu ^{**}}
and the claim follows if $|\mu ^{**}|\lec R^3$.
Suppose $|\mu ^{**}|\gg R^3$ and there are two different triplets $(n_1,n_2,n_3)$ satisfying the condition.
Let $a,b$ denote two different values for $n^*-n_1$, which are confined in the interval $[n^*-R,n^*+R]$ and satisfy $|a|,|b|\ge |\mu ^{**}|/(2R)$.
This time we use 
\eqq{\mathrm{lcm}(a,b)\mathrm{gcd}(a,b)=ab}
to deduce that
\eqq{|\mu ^{**}|\cdot 2R\ge (|\mu ^{**}|/(2R))^2,}
which contradicts the assumption $|\mu ^{**}|\gg R^3$.
Therefore, there is at most one choice in this case.

\eqref{est-d-}: We have
\eqq{(2n_1-n^*)\cdot n^*=\mu ^*,\qquad n^*\neq 0,\qquad |n_1-n_*|\le R,}
by which $n_1$ is restricted on the intersection of a hyperplane and a ball of radius $R$, and therefore the estimate follows.

\eqref{est-d'-1}: Let us assume $R_1\ge R_3$.
We rewrite the condition as
\eqs{n_{1,j}-n_{2,j}+n_{3,j}=n^*_j,\,(n_{1,j})^2-(n_{2,j})^2+(n_{3,j})^2=\mu ^*_j,\,|n_{1,j}|\le R_1,\,|n_{3,j}|\le R_3\quad (1\le j\le d),\\
\mu ^*_1+\mu ^*_2+\dots +\mu ^*_d=\mu ^*,\quad n_2\neq n_1,\quad n_2\neq n_3,}
which in particular yield
\eqq{2(n_{2,j}-n_{1,j})(n_{2,j}-n_{3,j})=(n_j^*)^2-\mu _j^*\qquad (1\le j\le d).}
Note that there is a freedom of choosing $\mu^*_j$'s under the condition $\mu ^*_1+\dots +\mu ^*_d=\mu ^*$.
For each $j$ and $\mu ^*_j$ fixed, the number of possible choices for $(n_{1,j},n_{2,j},n_{3,j})$ is estimated as follows:
\begin{enumerate}
\item If $\mu ^*_j\neq (n^*_j)^2$, then $n_{1,j}\neq n_{2,j}\neq n_{3,j}$ and the bound \eqref{est-1'-1} is applicable, obtaining $O(R_1^\eta )$ for any $\eta >0$.
\item If $\mu ^*_j=(n^*_j)^2$ and $n_{1,j}=n_{2,j}$, we have $n_{3,j}=n^*_j$ and estimate the number of such possibilities by $O(R_1)$.
\item If $\mu ^*_j=(n^*_j)^2$ and $n_{3,j}=n_{2,j}$, we obtain $O(R_3)$ similarly.
\end{enumerate}

By rearranging coordinates we may assume that $\mu ^*_j\neq (n^*_j)^2$ for $1\le j\le \nu$ and $\mu ^*_j=(n^*_j)^2$ for $\nu +1\le j\le d$, for some $0\le \nu \le d$.

If $\nu \ge 1$, for each $1\le j\le \nu -1$ we use the trivial bound $O(R_1R_3)$ on the number of possible $(n_{1,j},n_{2,j},n_{3,j})$, while for $\nu +1\le j\le d$ we invoke the bound in (ii) or (iii) above.
For $j=\nu$, observing that $\mu ^*_\nu$ is now determined by $n^*,\mu ^*$ and $(n_{1,j},n_{2,j},n_{3,j})_{j=1}^{\nu -1}$, we use the bound in (i).
In total, we obtain $O((R_1R_3)^{\nu -1}R_1^{d-\nu}R_1^\eta)$, which is maximized by the claimed one $O(R_1^{d-1}R_3^{d-1}R_1^{\eta})$ when $\nu =d$.

If $\nu =0$, the case (ii) or (iii) occurs for each $1\le j\le d$.
Note that (ii) cannot occur $d$ times; otherwise the condition $n_2\neq n_1$ would be violated.
Since $\mu ^*_j$'s are already fixed, we have an upper bound $O(R_1^{d-1}R_3)$, which is also smaller than the claimed one.

\eqref{est-d'-2}: Simply repeat the argument for \eqref{est-d'-1} using \eqref{est-1'-2} instead of \eqref{est-1'-1}.

\eqref{est-d'-}: Besides \eqref{est-d'-1} it suffices to take into account the case $n_2=n_1$ or $n_2=n_3$, which amounts to $O(\max \shugo{R_1,R_3}^d)$.
We thus obtain the claimed bound.
\end{proof}

\begin{rem}
Consider the bound \eqref{numberA}.
Since $d$-dimensional element $n$ is constrained by one equality $|n-n^*|^2=\mu ^*$, it is initially expected that the number of such $n$'s is of order at most $R^{d-1}$.
In this respect, \eqref{numberA} is better by almost one dimension than expected.
The same is true for all bounds in Lemma~\ref{lem:numbertheory}, Corollary~\ref{cor:number}, and Lemma~\ref{lem:number2}, except for the bound \eqref{est-d-} which is no better than the ``trivial'' one.
This fact has a large impact on the optimality of the regularity range in Theorem~\ref{thm:uniqueness}.
In fact, for Lemma~\ref{lem:trilinear} below we will not use \eqref{est-d-} so that we can cover the full subcritical range $s>s_c$, while for Lemma~\ref{lem:trilinear2} (only in the case $k=1$, $d=2,3$) we will have to rely on \eqref{est-d-}, resulting in non-optimal lower bounds of regularity in Theorem~\ref{thm:uniqueness} for these cases.
See also Remark~\ref{rem:linfty} below.
\end{rem}


\bigskip
\section{Proof of multilinear estimates}
\label{sec:proof}

This section is devoted to the proof of Proposition~\ref{prop:fundamental}.
Let us begin with the following:
\begin{lem}\label{lem:trilinear}
Let $d,k\in \Bo{N}$ with $(d,k)\neq (1,1)$, $s>s_c$.
Then, we have
\eqq{\sum _{\mat{n_0,n_1,n_2,\dots ,n_{2k+1}\in \Bo{Z}^d\\ n_0-n_1+\cdots -n_{2k+1}=0\\  |n_0|^2-|n_1|^2+\cdots -|n_{2k+1}|^2=\mu}}\prod _{l=0}^{2k+1}\om _l(n_l)~\lec~ N_{\max}^{-2s}\prod _{l=0}^{2k+1}N_l^{s}\tnorm{\om _l}{\ell ^2(\Bo{Z}^d)}}
for any $\mu \in \Bo{Z}$, $\shugo{N_l}_{l=0}^{2k+1}\subset \dyadic$, and any non-negative functions $\shugo{\om _l}_{l=0}^{2k+1}\subset \ell ^2(\Bo{Z}^d)$ satisfying $\mathrm{supp}~\om _l\subset \{ n:N_l\le \LR{n}<2N_l\}$, where $N_{\max}:=\max _{0\le l\le 2k+1}N_l$.
Here, the implicit constant is uniform in $\mu$ and $\{ N_l\}$.
\end{lem}

\begin{proof}
We may assume by symmetry that 
\eqq{N_0\ge N_2\ge \cdots \ge N_{2k},\qquad N_1\ge N_3\ge \cdots \ge N_{2k+1},\qquad N_0\ge N_1.}
Note that $N_{\max} =N_0\sim \max \{ N_1,N_2\}$.
Moreover, in the case $N_{\max}\sim N_{\mathrm{second}}\gg N_{\mathrm{third}}$, where $N_{\mathrm{second}}$ and $N_{\mathrm{third}}$ are the second and the third largest among $N_l$'s, we may restrict each of $\om _0$ and $\om _l$ corresponding to the largest and the second largest frequencies onto a ball of size $N_{\mathrm{third}}$ by almost orthogonality.

We fix $\mu \in \Bo{Z}$ and write ``$(*)$'' to denote the condition
\eqs{n_0-n_1+\dots -n_{2k+1}=0,\qquad |n_0|^2-|n_1|^2+\dots -|n_{2k+1}|^2=\mu .}
Let $S_1$, $S_2$ be two subsets of the index set $\shugo{0,1,\dots ,2k+1}$ such that $\# S_1,\# S_2\ge 2$ and $S_1\cap S_2=\emptyset$.
Let $S_3:=\shugo{0,1,\dots ,2k+1}\setminus (S_1\cup S_2)$, which may be empty.
Applying the Cauchy-Schwarz inequality several times, we have%
\footnote{If $S_3=\emptyset$, summation and supremum over $(n_m)_{m\in S_3}$ as well as product in $m\in S_3$ do not appear in the following calculation.}
\begin{align*}
&\sum _{\mat{n_0,n_1,\dots ,n_{2k+1}\\ (*)}}\prod _{l=0}^{2k+1}\om _l(n_l)~=~ \sum _{\mat{n_m\\ m\in S_3}}\prod _{m\in S_3}\om _m(n_m) \sum _{\mat{n_j\\ i\in S_2}}\prod _{j\in S_2}\om _j(n_j)\sum _{\mat{n_i;\,i\in S_1\\ (*)}}\prod _{i\in S_1}\om _i(n_i)\\
&\le \Big( \prod _{m\in S_3}N_m^{d/2}\tnorm{\om _m}{\ell ^2}\Big) \sup _{\mat{n_m\\ m\in S_3}} \sum _{\mat{n_j\\ j\in S_2}}\prod _{j\in S_2}\om _j(n_j)\cdot A_\mu ^{1/2}\Big( \sum _{\mat{n_i;\,i\in S_1\\ (*)}}\prod _{i\in S_1}\om _i(n_i)^2\Big) ^{1/2}
\intertext{where $A_\mu =A_\mu \big( (n_j)_{j\in S_2},(n_m)_{m\in S_3}\big) :=\# \Shugo{(n_i)_{i\in S_1}}{(*)}$,}
&\le \Big( \prod _{m\in S_3}N_m^{d/2}\tnorm{\om _m}{\ell ^2}\Big) \sup _{\mat{n_m\\ m\in S_3}} \Big( \sup _{\mat{n_j\\ j\in S_2}}A_\mu ^{1/2}\Big) \Big( \sum _{\mat{n_j\\ j\in S_2}}\prod _{j\in S_2}\om _j(n_j)^2\Big) ^{1/2}\Big( \sum _{\mat{n_j\\ j\in S_2}}\sum _{\mat{n_i;\,i\in S_1\\ (*)}}\prod _{i\in S_1}\om _i(n_i)^2\Big) ^{1/2}\\
&\le \Big( \prod _{m\in S_3}N_m^{d/2}\tnorm{\om _m}{\ell ^2}\Big) \Big( \prod _{j\in S_2}\tnorm{\om _j}{\ell ^2}\Big) \sup _{\mat{n_m\\ m\in S_3}} \Big( \sup _{\mat{n_j\\ j\in S_2}}A_\mu ^{1/2}\Big) \Big( \sum _{\mat{n_i\\ i\in S_1}}B_\mu \prod _{i\in S_1}\om _i(n_i)^2\Big) ^{1/2}
\intertext{where $B_\mu =B_\mu \big( (n_i)_{i\in S_1},(n_m)_{m\in S_3}\big) :=\# \Shugo{(n_j)_{j\in S_2}}{(*)}$,}
&\le \Big( \prod _{m\in S_3}N_m^{d/2}\tnorm{\om _m}{\ell ^2}\Big) \Big( \prod _{j\in S_2}\tnorm{\om _j}{\ell ^2}\Big) \Big( \prod _{i\in S_1}\tnorm{\om _i}{\ell ^2}\Big) \sup _{\mat{n_m\\ m\in S_3}} \Big( \sup _{\mat{n_j\\ j\in S_2}}A_\mu ^{1/2}\cdot \sup _{\mat{n_i\\ i\in S_1}}B_\mu ^{1/2}\Big) \\
&\le \prod _{m\in S_3}N_m^{d/2}\cdot \sup _{\mat{n_j;\,j\in S_2\\ n_m;\,m\in S_3}}A_\mu ^{1/2}\cdot \sup _{\mat{n_i;\,i\in S_1\\ n_m;\,m\in S_3}}B_\mu ^{1/2}\cdot \prod _{l=0}^{2k+1}\tnorm{\om _l}{\ell ^2}.
\end{align*}
Hence, it suffices to show, for $s>s_c$ and suitable $S_1,S_2$, that
\eq{claim1}{\sup _{\mat{n_j;\,j\in S_2\\ n_m;\,m\in S_3}}A_\mu \cdot \sup _{\mat{n_i;\,i\in S_1\\ n_m;\,m\in S_3}}B_\mu \cdot \prod _{m\in S_3}N_m^{d}~\lec~ \Big( N_{\max}^{-2}\prod _{l=0}^{2k+1}N_l\Big) ^{2s}.}

\smallskip
\underline{(I) $d\ge 2$, $k=1$.}
Take $S_1=\shugo{0,2}$, $S_2=\shugo{1,3}$, and thus $S_3=\emptyset$.
If $N_1\gec N_2$, then $N_0=N_{\max}\sim N_1$.
We use \eqref{est-d+} twice and obtain
\eqq{\sup _{n_1,n_3}A_\mu \cdot \sup _{n_0,n_2}B_\mu \lec N_2^{d-2+}N_3^{d-2+},}
which implies \eqref{claim1}.
If $N_2\gg N_1$ so that $N_0=N_{\max}\sim N_2$, we use the almost orthogonality to restrict $n_0$ and $n_2$ onto balls of size $N_1$, which yields the bound $N_1^{d-2+}N_3^{d-2+}$ and thus \eqref{claim1}.

\smallskip
\underline{(II) $d\ge 2$, $k\ge 2$.}

{\bf Case 1: $N_1\gec N_4$}.
Take $S_1=\shugo{0,2}$ and $S_2=\shugo{1,3}$.
In this case the same argument as (I) leads to the desired estimate.
We restrict $n_0,n_2$ as before if $N_2\sim N_{\max}\gg N_1$.
The estimate \eqref{est-d+} implies
\eqq{\text{LHS of \eqref{claim1}}\lec \min \shugo{N_1,N_2}^{d-2+}N_3^{d-2+}\prod _{l=4}^{2k+1}N_l^d\lec \Big( N_{\max}^{-2}\prod _{l=0}^{2k+1}N_l\Big) ^{d-\frac{2}{k}+},}
which is sufficient.

{\bf Case 2: $N_1\ll N_4$}.
In this case $N_0=N_{\max}\sim N_2$ and $N_4=N_{\mathrm{third}}$.
Take $S_1=\shugo{0,2}$, $S_2=\shugo{4,1,3}$, and restrict $n_0,n_2$ into $N_4$-balls if $N_2\gg N_4$.
For the estimate of $A_\mu$ and $B_\mu$ we apply \eqref{est-d+} and \eqref{est-d'-}, respectively.
We have
\eqq{\text{LHS of \eqref{claim1}}&\lec N_4^{d-2+}N_1^{d}N_3^{d-2+}\prod _{l=5}^{2k+1}N_l^d=\Big( N_4^{d-2+}N_1^d\prod_{i=3}^kN_{2i}^d\Big) \Big( N_3^{d-2+}\prod _{j=2}^kN_{2j+1}^d\Big) \\
&\lec \Big( N_{\max}^{-2}\prod _{l=0}^{2k+1}N_l\Big) ^{d-\frac{2}{k}+},}
which is also sufficient.

\smallskip
\underline{(III) $d=1$, $k\ge 2$.}
We take $S_1=\shugo{0,2,4}$ and $S_2=\shugo{1,3,5}$.
Dividing $n_0,n_2$ into $\max \shugo{N_1,N_4}$ scale if $N_2\sim N_{\max}\gg N_1$, we apply \eqref{est-1'+} twice to obtain
\eqq{\text{LHS of \eqref{claim1}}\lec \min \shugo{N_2,\max \shugo{N_1,N_4}}^{0+}N_3^{0+}\prod _{l=6}^{2k+1}N_l\lec \Big( N_{\max}^{-2}\prod _{l=0}^{2k+1}N_l\Big) ^{1-\frac{2}{k}+},}
as desired.

This concludes the proof.
\end{proof}

As a corollary, we obtain the following $\ell ^2$ estimate:
\begin{cor}\label{cor:trilinear}
Let $d,k\in \Bo{N}$ with $(d,k)\neq (1,1)$, $s>s_c$, and $-s\le s'\le s$.
Then, we have
\eqq{&\sup _{\mu \in \Bo{Z}}~\norm{\sum _{\mat{n_1,n_2,\dots ,n_{2k+1}\in \Bo{Z}^d\\ n=n_1-n_2+\dots +n_{2k+1},\, \Phi =\mu}}\prod _{l=1}^{2k+1}\om _l(n_l)}{\ell ^2_{s'}(\Bo{Z}^d_n)}~\lec~ \tnorm{\om _q}{\ell ^2_{s'}}\prod _{\mat{l=1\\ l\neq q}}^{2k+1}\tnorm{\om _l}{\ell ^2_s}}
for any $1\le q\le 2k+1$.
\end{cor}

\begin{proof}
We treat the case $q=1$; the same argument is applied to the other cases.
By duality, it suffices to show
\eqq{\bigg| \sum _{\mat{n_0,n_1,n_2,\dots ,n_{2k+1}\in \Bo{Z}^d\\ n_0-n_1+\cdots -n_{2k+1}=0\\  |n_0|^2-|n_1|^2+\cdots -|n_{2k+1}|^2=\mu}}\prod _{l=0}^{2k+1}\om _l(n_l)\bigg| \le C\tnorm{\om _0}{\ell ^2_{-s'}}\tnorm{\om _1}{\ell ^2_{s'}}\prod _{l=2}^{2k+1}\tnorm{\om _l}{\ell ^2_s}.}
Choose $\e >0$ so that $s-\e >s_c$.
By Lemma~\ref{lem:trilinear}, the left-hand side is bounded by
\eqq{
&\sum _{N_0,\dots ,N_{2k+1}\in \dyadic}\sum _{\mat{n_0,n_1,n_2,\dots ,n_{2k+1}\in \Bo{Z}^d\\ n_0-n_1+\cdots -n_{2k+1}=0\\  |n_0|^2-|n_1|^2+\cdots -|n_{2k+1}|^2=\mu}}\prod _{l=0}^{2k+1}\big| P_{N_l}\om _l(n_l)\big| \\
&\lec \sum _{N_0,\dots ,N_{2k+1}}\Big( \frac{N_0N_1\cdots N_{2k+1}}{N_{\max}^2}\Big) ^{s-\e}\prod _{l=0}^{2k+1}\norm{P_{N_l}\om _l}{\ell ^2}\\
&\lec \sum _{N_0,\dots ,N_{2k+1}}\Big( \frac{N_0N_1\cdots N_{2k+1}}{N_{\max}^2}\Big) ^{-\e}\norm{P_{N_0}\om _0}{\ell ^2_{-s'}}\norm{P_{N_1}\om _1}{\ell ^2_{s'}}\prod _{l=2}^{2k+1}\norm{P_{N_l}\om _l}{\ell ^2_s},
}
where $P_N\om (n):=\chf{\{ N\le \LR{n}<2N\}}\om (n)$, and at the last inequality we have used the fact that $N_0^sN_1^s/N_{\max}^{2s}\le N_0^{-s'}N_1^{s'}$ for any $-s\le s'\le s$.
The factor $(N_0N_1\cdots N_{2k+1}/N_{\max}^2)^{-\e}$ is enough for summing up over $N_0,N_1,\dots N_{2k+1}$:
For $N_{\max}$ and $N_{\mathrm{second}}\sim N_{\max}$ we can use Cauchy-Schwarz by orthogonality, and for the others there is a negative power of $N_l$. 
Therefore, the claim follows.
\end{proof}

We will also use the $\ell^\I$ estimate below:
\begin{lem}\label{lem:trilinear2}
Let $d,k\in \Bo{N}$ be such that $(d,k)\neq (1,1)$ and $d<2+\frac{2}{k}$.
Assume that $s\in \R$ satisfies 
\eqq{\begin{cases}
s>\tfrac{d}{2}-\tfrac{1}{2} &\text{if $k=1$ and $d=2,3$,}\\
\text{$s>s_c$ and $s>s_e-\e (k)$} &\text{if $k\ge 2$ and $d=1,2$,}
\end{cases}}
where $\e (k):=\frac{1}{2}\min \{ \frac{1}{k(2k+1)},\,\frac{3}{5}-\frac{9}{16},\,\frac{2k+3}{4k(2k+1)},\,\frac{3}{10}-\frac{1}{6}\} >0$.
Then, we have
\eqq{\sup _{\mu \in \Bo{Z}}~\norm{\sum _{\mat{n_1,\dots ,n_{2k+1}\in \Bo{Z}^d\\ n=n_1-n_2+\dots +n_{2k+1}}}\chf{\Sc{A}^c\cap \{ \Phi =\mu \}}\prod _{l=1}^{2k+1}\om _l(n_l)}{\ell ^\I (\Bo{Z}^d_n)}\lec \norm{\om _q}{\ell ^\I}\prod _{\mat{l=1\\ l\neq q}}^{2k+1}\norm{\om _l}{\ell ^2_s}}
for any $1\le q\le 2k+1$.
\end{lem}

\begin{proof}
It suffices to prove for
\eqq{s_*=s_*(d,k):=\begin{cases}
\tfrac{d}{2}-\tfrac{1}{2} &\text{if $k=1$ and $d=2,3$,}\\
\max \shugo{s_c+,s_e-\e} &\text{if $k\ge 2$ and $d=1,2$}
\end{cases}
}
that
\eq{claim2-}{\norm{\sum _{\mat{n_1,\dots ,n_{2k+1}\in \Bo{Z}^d\\ n=n_1-n_2+\dots +n_{2k+1}}}\chf{\Sc{A}^c\cap \{ \Phi =\mu \}}\ti{\om}(n_q)\prod _{\mat{l=1\\ l\neq q}}^{2k+1}\om _l(n_l)}{\ell ^\I (\Bo{Z}^d_n)}\lec \norm{\ti{\om}}{\ell ^\I}\prod _{\mat{l=1\\ l\neq q}}^{2k+1}N_l^{s_*}\norm{\om _l}{\ell ^2}}
for any $\mu \in \Bo{Z}$, $1\le q\le 2k+1$, $(N_l)_{l=1}^{2k+1}\subset \dyadic$, and any non-negative functions $\shugo{\om _l}_{l=1}^{2k+1}\subset \ell ^2(\Bo{Z}^d)$ satisfying $\mathrm{supp}~\om _l\subset \{ n:N_l\le \LR{n}<2N_l\}$.
Note that $n_q$ is not restricted to a dyadic region.
In the following calculation we denote by $(*)$ the condition (for given $q$)
\eqs{n=n_1-n_2+\dots +n_{2k+1},\qquad \Phi =\mu ,\qquad (n_l)_{l=1}^{2k+1}\not\in \Sc{A},\\
N_l\le \LR{n_l}<2N_l\qquad (1\le l\le 2k+1,\,l\neq q).}

Given $q$, let $T_1,T_2$ be subsets of $\shugo{1,2,\dots ,2k+1}$ such that $\# T_1,\# T_2\ge 2$, $T_1\cap T_2=\shugo{q}$, and let $T_3:=\shugo{1,2,\dots ,2k+1}\setminus (T_1\cup T_2)$.
Note that $T_3$ may be empty.
By the Cauchy-Schwarz inequality, we have
\begin{align*}
&\norm{\sum _{\mat{n_1,\dots ,n_{2k+1}\\ (*)}}\ti{\om}(n_q)\prod _{\mat{l=1\\ l\neq q}}^{2k+1}\om _l(n_l)}{\ell ^\I}\\
&\le \tnorm{\ti{\om}}{\ell ^\I}\sup _{n}\sum _{\mat{n_m\\ m\in T_3}}\prod _{m\in T_3}\om _m(n_m)\sum _{\mat{n_j\\ j\in T_2\setminus \shugo{q}}}\prod _{j\in T_2\setminus \shugo{q}}\om _j(n_j)\sum _{\mat{n_i;\,i\in T_1\\ (*)}}\prod _{i\in T_1\setminus \shugo{q}}\om _i(n_i)\\
&\le \tnorm{\ti{\om}}{\ell ^\I}\Big( \prod _{m\in T_3}N_m^{d/2}\tnorm{\om _m}{\ell ^2}\Big) \sup _{\mat{n,n_m\\ m\in T_3}}\sum _{\mat{n_j\\ j\in T_2\setminus \shugo{q}}}\prod _{j\in T_2\setminus \shugo{q}}\om _j(n_j)\sum _{\mat{n_i;\,i\in T_1\\ (*)}}\prod _{i\in T_1\setminus \shugo{q}}\om _i(n_i)\\
&\le \tnorm{\ti{\om}}{\ell ^\I}\Big( \prod _{m\in T_3}N_m^{d/2}\tnorm{\om _m}{\ell ^2}\Big) \sup _{\mat{n,n_m\\ m\in T_3}}\Big( \prod _{j\in T_2\setminus \shugo{q}}\tnorm{\om _j}{\ell ^2}\Big) \Big( \sum _{\mat{n_j\\ j\in T_2\setminus \shugo{q}}}\Big( \sum _{\mat{n_i;\,i\in T_1\\ (*)}}\prod _{i\in T_1\setminus \shugo{q}}\om _i(n_i)\Big) ^2\Big) ^{1/2}\\
&\le \tnorm{\ti{\om}}{\ell ^\I}\Big( \prod _{m\in T_3}N_m^{d/2}\tnorm{\om _m}{\ell ^2}\Big) \Big( \prod _{j\in T_2\setminus \shugo{q}}\tnorm{\om _j}{\ell ^2}\Big) \sup _{\mat{n,n_m\\ m\in T_3}}\Big( \sum _{\mat{n_j\\ j\in T_2\setminus \shugo{q}}}A_\mu '\sum _{\mat{n_i;\,i\in T_1\\ (*)}}\prod _{i\in T_1\setminus \shugo{q}}\om _i(n_i)^2\Big) ^{1/2},
\intertext{where $A'_\mu =A'_\mu \big( n,(n_l)_{l\in T_2\cup T_3\setminus \{ q\}}\big) :=\# \Shugo{(n_i)_{i\in T_1}}{(*)}$,}
&\le \tnorm{\ti{\om}}{\ell ^\I}\Big( \prod _{m\in T_3}N_m^{d/2}\tnorm{\om _m}{\ell ^2}\Big) \Big( \hz\prod _{j\in T_2\setminus \shugo{q}}\hz \tnorm{\om _j}{\ell ^2}\Big) \sup _{\mat{n,n_m\\ m\in T_3}}\Big( \hz \sup _{\mat{n_j\\ j\in T_2\setminus \{ q\}}}\hz\hz A'_\mu \cdot \hz \sum _{\mat{n_i\\ i\in T_1\setminus \shugo{q}}}\sum _{\mat{n_j;\,j\in T_2\\ (*)}}\prod _{i\in T_1\setminus \shugo{q}}\hz \om _i(n_i)^2\Big) ^{1/2}\\
&\le \tnorm{\ti{\om}}{\ell ^\I}\Big( \prod _{m\in T_3}N_m^{d/2}\tnorm{\om _m}{\ell ^2}\Big) \Big( \hz\prod _{j\in T_2\setminus \shugo{q}}\hz \tnorm{\om _j}{\ell ^2}\Big) \Big( \hz\prod _{i\in T_1\setminus \shugo{q}}\hz \tnorm{\om _i}{\ell ^2}\Big) \sup _{\mat{n,n_m\\ m\in T_3}}\Big( \hz \sup _{\mat{n_j\\ j\in T_2\setminus \{ q\}}}\hz A'_\mu \cdot \hz \sup _{\mat{n_i\\ i\in T_1\setminus \{ q\}}}\hz B'_\mu \Big) ^{1/2},
\end{align*}
where $B'_\mu =B'_\mu \big( n,(n_l)_{l\in T_1\cup T_3\setminus \{ q\}}\big) :=\# \Shugo{(n_j)_{j\in T_2}}{(*)}$.
Hence, it suffices to show that
\eq{claim2}{\sup _{\mat{n,n_l\\ l\not\in T_1}}\# \Shugo{(n_i)_{i\in T_1}}{(*)} \cdot \sup _{\mat{n,n_l\\ l\not\in T_2}}\# \Shugo{(n_j)_{j\in T_2}}{(*)} \cdot \prod _{m\in T_3}N_m^d \lec \prod _{\mat{l=1\\ l\neq q}}^{2k+1}N_l^{2s_*}}
for any $q\in \shugo{1,2,\dots ,2k+1}$ with an appropriate choice of $T_1,T_2$.
Moreover, in the case $k\ge 2$ we will need to fix the order of size of $|n_1|,\dots ,|n_{2k+1}|$.
We see that 
\eqq{\text{LHS of \eqref{claim2-}} \le \sum _{\sigma \in \mathfrak{S}_{2k+1}}\norm{\sum _{\mat{n_1,\dots ,n_{2k+1}\\ (*)}}\chf{\{ n_{\sigma (1)}\succeq n_{\sigma (2)}\succeq \cdots \succeq n_{\sigma (2k+1)}\}}\ti{\om}(n_q)\prod _{\mat{l=1\\ l\neq q}}^{2k+1}\om _l(n_l)}{\ell ^\I (\Bo{Z}^d_n)},}
where $\mathfrak{S}_{2k+1}$ denotes the symmetric group of degree $2k+1$.
By the same argument as above, it also suffices to show that
\eq{claim2'}{&\sup _{\mat{n,n_l\\ l\not\in T_1}}\# \Shugo{(n_i)_{i\in T_1}}{(*),\,n_{\sigma (1)}\succeq \cdots \succeq n_{\sigma (2k+1)}} \\[-15pt]
&\qquad \times \sup _{\mat{n,n_l\\ l\not\in T_2}}\# \Shugo{(n_j)_{j\in T_2}}{(*),\,n_{\sigma (1)}\succeq \cdots \succeq n_{\sigma (2k+1)}} \cdot \prod _{m\in T_3}N_m^d \lec \prod _{\mat{l=1\\ l\neq q}}^{2k+1}N_l^{2s_*}}
for any $q\in \shugo{1,2,\dots ,2k+1}$ and $\sigma \in \mathfrak{S}_{2k+1}$, with suitable $T_1,T_2$.

\smallskip
\underline{(I) $k=1$ and $d=2,3$.}
The desired bound will be obtained from the estimates \eqref{est-d+} and \eqref{est-d-}.
We only consider the worst case $q=2$: Take $T_1=\shugo{1,2}$, $T_2=\shugo{2,3}$, then $T_3=\emptyset$ and two applications of \eqref{est-d-} yield
\eqq{\sup _{n,n_3}\# \Shugo{(n_1,n_2)\in (\Bo{Z}^d)^2}{(*)}\cdot \sup _{n,n_1}\# \Shugo{(n_2,n_3)\in (\Bo{Z}^d)^2}{(*)}\lec N_1^{d-1}N_3^{d-1},}
which verifies \eqref{claim2'}.

\smallskip
\underline{(II) $k\ge 2$ and $d=2$.}
We verify \eqref{claim2'} considering several cases separately according to the choice of $q,\sgm$.
Recall that $n_{[m]}$ is the $m$-th largest among $(n_l)_{l=1}^{2k+1}$ in the order $\preceq$.
For given $\sigma \in \mathfrak{S}_{2k+1}$, we denote by $[m]$ the index $l$ such that $n_{[m]}=n_{l}$; i.e., $[m]:=\sgm ^{-1}(m)$.

{\bf Case 1: $n_q=n_{[1]}$}.
We take $T_1=\shugo{q,[2],[3]}$ and $T_2=\shugo{q,[4],[5]}$. 
By the definition of the exceptional set $\Sc{A}$, we have $n_q\neq n_{[2]}, n_{[3]}, n_{[4]}, n_{[5]}$ and $n_{[2]}\neq n_{[3]}$.
Using one of the estimates \eqref{est-d'+}, \eqref{est-d'-1}, and \eqref{est-d'-2}, we obtain that
\eqq{\sup _{n,n_l;\,l\not\in T_1}\# \Shugo{(n_i)_{i\in T_1}}{(*)}\lec N_{[2]}^{1+}N_{[3]}.}
These estimates also imply
\eqq{\sup _{n,n_l;\,l\not\in T_2}\# \Shugo{(n_j)_{j\in T_2}}{(*)}\lec N_{[4]}^{1+}N_{[5]}}
if $n_{[4]}\neq n_{[5]}$, while we have
\eqq{\sup _{n,n_l;\,l\not\in T_2}\# \Shugo{(n_j)_{j\in T_2}}{(*)}\lec N_{[5]}^2\le N_{[4]}N_{[5]}}
under the additional assumption $n_{[4]}=n_{[5]}$.
Consequently, the left hand side of \eqref{claim2'} is bounded by
\eqq{N_{[2]}^{1+}N_{[3]}N_{[4]}^{1+}N_{[5]}\prod _{l=6}^{2k+1}N_{[l]}^2\le \prod _{l=2}^{2k+1}N_{[l]}^{2s_c+},}
which is favorable since $s_e=s_c+\frac{1}{k(2k+1)}>s_c$ in this case.

{\bf Case 2: $n_{[1]}\neq n_q=n_{[2]}$}.
We take $T_1=\shugo{[1],q,[3]}$, $T_2=\shugo{q,[4],[5]}$.
The same argument as Case 1 yields that
\eqq{\text{LHS of \eqref{claim2'}}&\lec N_{[1]}^{1+}N_{[3]}N_{[4]}^{1+}N_{[5]}\prod _{l=6}^{2k+1}N_{[l]}^2\lec \Big( N_{[1]}\prod _{l=3}^{2k+1}N_{[l]}\Big) ^{2s_c+}.}

{\bf Case 3: $n_{[2]}\neq n_q=n_{[3]}$}.
This is the only delicate case.
We take $T_1=\shugo{[1],[2],q}$ and $T_2=\shugo{q,[4],[5]}$, but now both $n_q=n_{[4]}$ and $n_{[4]}=n_{[5]}$ are possible to occur.
When $n_q=n_{[4]}$, noticing $|n_q|\le 2N_{[4]}$, the same argument as above implies that
\eqq{\text{LHS of \eqref{claim2'}}\lec N_{[2]}^{1+}N_{[4]}\cdot N_{[4]}^2\prod _{l=6}^{2k+1}N_{[l]}^2.}
If we further assume that $k\ge 3$, this is bounded by
\eqq{&N_{[2]}^{1+}N_{[4]}\cdot N_{[4]}^2\prod _{l=6}^{k+2}N_{[l]}^2\cdot \prod _{l=k+3}^{2k+1}N_{[l]}^2\le \Big( N_{[1]}N_{[2]}N_{[4]}\prod _{l=6}^{k+2}N_{[l]}\Big) ^{2s_c+} \Big( N_{[5]}\prod _{l=k+3}^{2k+1}N_{[l]}\Big) ^{2s_c},}
which is sufficient for the claim.
For $k=2$, however, the resulting bound is $(N_{[1]}N_{[2]}N_{[4]})^{\frac{4}{3}+}$, which is not acceptable by $(\frac{2}{3}+)>s_*=(\frac{3}{5}-)$.
We now make use of the additional property $\LR{n_{[2]}}>\LR{n_q}^{3/2}$ of $\Sc{A}^c$.
Since $N_{[4]}\le \LR{n_q}<(2N_{[2]})^{2/3}$, we have
\eqq{\text{LHS of \eqref{claim2'}}\lec N_{[2]}^{1+}N_{[4]}\cdot N_{[4]}^2\lec N_{[2]}^{\frac{9}{4}+}N_{[4]}^{\frac{9}{8}}\le  (N_{[1]}N_{[2]}N_{[4]})^{\frac{9}{8}+}.}
This is sufficient, because $\frac{9}{8}<\frac{6}{5}=2s_e$.
For the remaining cases (\mbox{i.e.}, $n_q\neq n_{[4]}=n_{[5]}$ or $n_q\neq n_{[4]}\neq n_{[5]}$), we treat just as Case 1 and obtain
\eqq{\text{LHS of \eqref{claim2'}}&\lec N_{[2]}^{2+}\cdot N_{[4]}^{1+}N_{[5]}\prod _{l=6}^{2k+1}N_{[l]}^2\lec \Big( N_{[1]}N_{[2]}\prod _{l=4}^{2k+1}N_{[l]}\Big) ^{2s_c+}.}

{\bf Case 4: $n_{[3]}\neq n_q=n_{[4]}$}.
We take $T_1=\shugo{[1],[2],q}$, $T_2=\shugo{[3],q,[5]}$.
In this case $n_q=n_{[5]}$ is possible.
The same argument as Case 1 with $|n_q|<2N_{[3]}$ implies
\eqq{\text{LHS of \eqref{claim2'}}&\lec N_{[2]}^{1+}N_{[3]}\cdot N_{[3]}^{1+}N_{[5]}\prod _{l=6}^{2k+1}N_{[l]}^2\lec \Big( N_{[1]}N_{[2]}N_{[3]}\prod _{l=5}^{2k+1}N_{[l]}\Big) ^{2s_c+},}
as desired.

{\bf Case 5: $n_q\neq n_{[1]},n_{[2]},n_{[3]},n_{[4]}$}.
We take $T_1=\shugo{[1],[2],q}$, $T_2=\shugo{[3],[4],q}$.
In this case $n_{[3]}=n_{[4]}$ is possible, and the same argument as Case 1 with $|n_q|<2N_{[4]}$ implies
\eqq{\text{LHS of \eqref{claim2'}}&\lec N_{[2]}^{1+}N_{[4]}\cdot N_{[4]}^{2+}\prod _{\mat{l=5\\ l\neq q}}^{2k+1}N_{[l]}^2\lec \Big( \prod _{\mat{l=1\\ l\neq q}}^{2k+1}N_{[l]}\Big) ^{2s_c+},}
which is sufficient.

\smallskip
\underline{(III) $k\ge 2$ and $d=1$.}
We follow the argument in (II) using \eqref{est-1'+}, \eqref{est-1'-1}, and \eqref{est-1'-2} instead of \eqref{est-d'+}, \eqref{est-d'-1}, and \eqref{est-d'-2}, respectively.

{\bf Case 1: $n_q=n_{[1]}$} ($T_1=\shugo{q,[2],[3]}$, $T_2=\shugo{q,[4],[5]}$).
Taking into account the case $n_{[4]}=n_{[5]}$, we have
\eqq{\text{LHS of \eqref{claim2'}}&\lec N_{[2]}^{0+}(N_{[4]}^{0+}+N_{[5]})\prod _{l=6}^{2k+1}N_{[l]}\le \prod _{l=2}^{2k+1}N_{[l]}^{1-\frac{3}{2k}+}.}
This is sufficient, since $2s_e -(1-\frac{3}{2k})=\frac{2k+3}{2k(2k+1)}>0$.

{\bf Case 2: $n_{[1]}\neq n_q=n_{[2]}$} ($T_1=\shugo{[1],q,[3]}$, $T_2=\shugo{q,[4],[5]}$).
In the same manner, we have 
\eqq{\text{LHS of \eqref{claim2'}}&\lec N_{[1]}^{0+}(N_{[4]}^{0+}+N_{[5]})\prod _{l=6}^{2k+1}N_{[l]}\le \Big( N_{[1]}\prod _{l=3}^{2k+1}N_{[l]}\Big) ^{1-\frac{3}{2k}+}.}

{\bf Case 3: $n_{[2]}\neq n_q=n_{[3]}$} ($T_1=\shugo{[1],[2],q}$, $T_2=\shugo{q,[4],[5]}$).
In contrast to the case of $d=2$, a crude estimate suffices for any $k\ge 2$.
We have
\eqq{\text{LHS of \eqref{claim2'}}&\lec N_{[2]}^{0+}(N_{[4]}^{0+}+N_{[4]})\prod _{l=6}^{2k+1}N_{[l]}\le \begin{cases}
(N_{[1]}N_{[2]}N_{[4]})^{\frac{1}{3}+}&\text{if $k=2$},\\[5pt]
\Big( N_{[1]}N_{[2]}\displaystyle\prod\limits _{l=4}^{2k+1}N_{[l]}\Big) ^{1-\frac{3}{2k}+}&\text{if $k\ge 3$},
\end{cases}}
which is sufficient since $\frac{1}{3}<\frac{3}{5}=2s_e$ if $k=2$.

{\bf Case 4: $n_{[3]}\neq n_q=n_{[4]}$} ($T_1=\shugo{[1],[2],q}$, $T_2=\shugo{[3],q,[5]}$).
We have
\eqq{\text{LHS of \eqref{claim2'}}&\lec N_{[2]}^{0+}(N_{[3]}^{0+}+N_{[5]})\prod _{l=6}^{2k+1}N_{[l]}\le \Big( N_{[1]}N_{[2]}N_{[3]}\prod _{l=5}^{2k+1}N_{[l]}\Big) ^{1-\frac{3}{2k}+}.}

{\bf Case 5: $n_q\neq n_{[1]},n_{[2]},n_{[3]},n_{[4]}$} ($T_1=\shugo{[1],[2],q}$, $T_2=\shugo{[3],[4],q}$).
We have
\eqq{\text{LHS of \eqref{claim2'}}&\lec N_{[2]}^{0+}(N_{[4]}^{0+}+N_{[4]})\prod _{\mat{l=5\\ l\neq q}}^{2k+1}N_{[l]}\lec \Big( \prod _{\mat{l=1\\ l\neq q}}^{2k+1}N_{[l]}\Big) ^{1-\frac{3}{2k}+}.}
This completes the proof.
\end{proof}

\begin{rem}\label{rem:linfty}
When $k=1$ and $d=2,3$, there is still some gap between the regularity threshold $\frac{d}{2}-\frac{1}{2}$ obtained in Lemma~\ref{lem:trilinear2} and the expected one $\max \{ s_c,s_e\} =s_e=\frac{d}{6}$.
In the case $d=2$, it turns out that the claimed estimate actually fails for $s<\frac{1}{2}$; this can be easily seen by testing with $\om _1=\chf{\{ (n^1,0):|n^1|\le N\}}$, $\om _2=\chf{\{ (n^1,n^2):\max (|n^1|,|n^2|)\le N\}}$, $\om _3=\chf{\{ (0,n^2):|n^2|\le N\}}$, $q=2$, $\mu =0$, $n=0$, and taking $N\to \I$.
The threshold might be improved by some further analysis in the case $d=3$.
\end{rem}

We are now ready to give a proof of Proposition~\ref{prop:fundamental}.
\begin{proof}[Proof of Proposition~\ref{prop:fundamental}]
\underline{Estimate (B1)}.
This is a special case ($s'=s$) of Corollary~\ref{cor:trilinear}.

\underline{Estimate (B1)'}.
This follows from the Sobolev inequality, since $s_2>d/2$.
 
\underline{Estimate (R)}. 
If $k=1$, then we have
\eqq{\bigg| \sum _{\mat{n_1,n_2,n_3\in \Bo{Z}^d\\ n=n_1-n_2+n_3}}\chf{\Sc{A}}\prod _{l=1}^{3}\om _l(n_l)\bigg| \le \Big[ \sum _{\mat{n=n_1-n_2+n_3\\ n_2=n_1}}+\sum _{\mat{n=n_1-n_2+n_3\\ n_2=n_3}}\Big] |\om _1(n_1)\om _2(n_2)\om _3(n_3)|.}
By the Cauchy-Schwarz inequality, for $s\ge 0$ we have
\eqq{\text{LHS of (R)}\le \tnorm{\om _1}{\ell ^2}\tnorm{\om _2}{\ell ^2}\tnorm{\om _3}{\ell ^2_s}+\tnorm{\om _1}{\ell ^2_s}\tnorm{\om _2}{\ell ^2}\tnorm{\om _3}{\ell^2}\le 2\prod _{l=1}^3\tnorm{\om _l}{\ell ^2_s}.}

For $k\ge 2$, we have
\eqq{\bigg| \sum _{\mat{n_1,n_2,\dots ,n_{2k+1}\in \Bo{Z}^d\\ n=n_1-n_2+\dots +n_{2k+1}}}\chf{\Sc{A}}\prod _{l=1}^{2k+1}\om _l(n_l)\bigg| \le \sum _{\sgm \in \mathfrak{S}_{2k+1}}\sum _{i=1}^3\sum _{\mat{(n_l)_{l=1}^{2k+1}\in \Sc{A}_i,\,n_{[l]}=n_{\sgm (l)}\\n=n_1-n_2+\dots +n_{2k+1}}}\prod _{l=1}^{2k+1}|\om _{[l]}(n_{[l]})|.}
Let us focus on the case $\sgm =\mathrm{id}$ for simplicity.
The terms for $i=1$ are treated with Young's inequality:
\eqq{&\norm{\LR{n}^s\sum _{n_1\in \Bo{Z}^d}|\om _1(n_1)\om _2(n_1)| \sum _{n=n_3-n_4+\dots +n_{2k+1}}\prod _{l=3}^{2k+1}|\om _l(n_l)|}{\ell ^2}\\
&\lec \norm{\sum _{n_1\in \Bo{Z}^d}\LR{n_1}^{s+\frac{d(k-1)}{2k}}|\om _1(n_1)\om _2(n_1)| \sum _{n=n_3-n_4+\dots +n_{2k+1}}\LR{n_3}^{\frac{d(k-1)}{2k}}|\om _3(n_3)|\prod _{l=4}^{2k+1}\LR{n_l}^{-\frac{d}{2k}}|\om _l(n_l)|}{\ell ^2}\\
&\le \tnorm{\om _1}{\ell ^2_s}\tnorm{\om _2}{\ell ^2_{\frac{d(k-1)}{2k}}}\tnorm{\om _3}{\ell ^2_{\frac{d(k-1)}{2k}}}\prod _{l=4}^{2k+1}\tnorm{\om _l}{\ell ^1_{-\frac{d}{2k}}}.}
This is sufficient, since $\frac{d(k-1)}{2k}=-\frac{d}{2k}+\frac{d}{2}<-\frac{d}{2k+1}+\frac{d}{2}=s_e$.
The case $i=2$ is treated in almost the same manner.
For $i=3$, we exploit the restriction $\LR{n_2}\le \LR{n_3}^{3/2}$ to obtain
\eqq{&\norm{\LR{n}^{s}\sum _{n_3\in \Bo{Z}^d}|\om _3(n_3)\om _4(n_3)|\sum _{n=n_1-n_2+n_5-\dots +n_{2k+1}}|\om _1(n_1)\om _2(n_2)|\prod _{l=5}^{2k+1}|\om _l(n_l)|}{\ell ^2}\\
&\lec \big\| \sum _{n_3\in \Bo{Z}^d}\LR{n_3}^{\frac{d(2k-3/2)}{2k+1/2}}|\om _3(n_3)\om _4(n_3)|\\[-15pt]
&\qquad\quad \times \sum _{n=n_1-n_2+n_5-\dots +n_{2k+1}}\LR{n_1}^{s}|\om _1(n_1)|\LR{n_2}^{-\frac{d}{2k+1/2}}|\om _2(n_2)|\prod _{l=5}^{2k+1}\LR{n_l}^{-\frac{d}{2k+1/2}}|\om _l(n_l)| \big\| _{\ell ^2}\\
&\le \tnorm{\om _1}{\ell ^2_s}\tnorm{\om _3}{\ell ^2_{\frac{d(2k-3/2)}{2(2k+1/2)}}}\tnorm{\om _4}{\ell ^2_{\frac{d(2k-3/2)}{2(2k+1/2)}}}\tnorm{\om _2}{\ell ^1_{-\frac{d}{2k+1/2}}}\prod _{l=5}^{2k+1}\tnorm{\om _l}{\ell ^1_{-\frac{d}{2k+1/2}}}.}
This is also sufficient, because $\frac{d(2k-3/2)}{2(2k+1/2)}=-\frac{d}{2k+1/2}+\frac{d}{2}<-\frac{d}{2k+1}+\frac{d}{2}=s_e$.

\underline{Estimates (B2), (B2)', (B3)}.
Consider the following three cases separately: (i) $d\ge 2+\frac{2}{k}$, (ii) $d=1,2$ and $k\ge 2$, (iii) $d=2,3$ and $k=1$.

For (i), we use Corollary~\ref{cor:trilinear} with $\sgm =-s_c$ to obtain (B2).
The estimate (B2)' is verified by the Sobolev inequalities
\eqq{\tnorm{fg}{H^{-s_c}}\lec \tnorm{f}{H^{-s_c}}\tnorm{g}{H^{s_2}},\qquad \tnorm{fg}{H^{s_2}}\lec \tnorm{f}{H^{s_2}}\tnorm{g}{H^{s_2}}.}
Note that $0\le s_c<\frac{d}{2}<s_2$.
For (B3), we use the Sobolev embeddings
\eqq{\tnorm{f}{H^{-s_c}}\lec \tnorm{f}{L^{\frac{dk}{dk-1}}},\qquad \tnorm{f}{L^{\frac{dk(2k+1)}{dk-1}}}\lec \tnorm{f}{H^{s_e+\frac{1}{k(2k+1)}}},}
and note that $s_c=s_e+\frac{1}{2k+1}(d-2-\frac{1}{k})\ge s_e+\frac{1}{k(2k+1)}$ if $d\ge 2+\frac{2}{k}$.

For (ii), Lemma~\ref{lem:trilinear2} gives (B2).
(B2)' follows from the inequality $\tnorm{\phi *\psi}{\ell ^\I}\lec \tnorm{\phi}{\ell ^\I}\norm{\psi}{\ell ^2_{s_2}}$ for any $s_2>d/2$.
Since $s\ge s_e$ and $\tnorm{f}{L^{2k+1}}\lec \tnorm{f}{H^{s_e}}$, (B3) holds.

For (iii), neither Corollary~\ref{cor:trilinear} nor Lemma~\ref{lem:trilinear2} is sufficient, so we interpolate these estimates to optimize the regularity range.
Let
\eqq{M(\om _1,\om _2,\om _3)(n)&:=\sum _{n=n_1-n_2+n_3}\chf{\{ n_1\neq n_2\neq n_3\}}\om _1(n_1)\om _2(n_2)\om _3(n_3),\\
M_\mu (\om _1,\om _2,\om _3)(n)&:=\sum _{n=n_1-n_2+n_3}\chf{\{ n_1\neq n_2\neq n_3\} \cap \{ \Phi =\mu \}}\om _1(n_1)\om _2(n_2)\om _3(n_3)\qquad (\mu \in \Bo{Z}).}
We have already seen in (i), (ii) that
\eqq{\norm{M(\om _1,\om _2,\om _3)}{\ell ^2_{-(\frac{d}{2}-1)}}\lec \prod _{l=1}^3\tnorm{\om _l}{\ell ^2_{s_e+\frac{1}{3}}},\qquad \norm{M(\om _1,\om _2,\om _3)}{\ell ^\I}\lec \prod _{l=1}^3\tnorm{\om _l}{\ell ^2_{s_e}},}
and also that
\eqq{\norm{M(\om _1,\om _2,\om _3)}{\ell ^2_{-(\frac{d}{2}-1)}}\lec \tnorm{\om _q}{\ell ^2_{-(\frac{d}{2}-1)}}\prod _{\mat{l=1\\ l\neq q}}^3\tnorm{\om _l}{\ell ^2_{s_2}},\quad \norm{M(\om _1,\om _2,\om _3)}{\ell ^\I}\lec \tnorm{\om _q}{\ell ^\I}\prod _{\mat{l=1\\ l\neq q}}^3\tnorm{\om _l}{\ell ^2_{s_2}},
}
while Corollary~\ref{cor:trilinear} and Lemma~\ref{lem:trilinear2} give
\eqq{
\norm{M_\mu (\om _1,\om _2,\om _3)}{\ell ^2_{-(\frac{d}{2}-1)}}&\lec \tnorm{\om _q}{\ell ^2_{-(\frac{d}{2}-1)}}\prod _{\mat{l=1\\ l\neq q}}^3\tnorm{\om _l}{\ell ^2_{\frac{d}{2}-1+}},\\
\norm{M_\mu (\om _1,\om _2,\om _3)}{\ell ^\I}&\lec \tnorm{\om _q}{\ell ^\I}\prod _{\mat{l=1\\ l\neq q}}^3\tnorm{\om _l}{\ell ^2_{\frac{d}{2}-\frac{1}{2}+}}\qquad (\mu \in \Bo{Z},~1\le q\le 3).}
Interpolating these estimates, we have
\eqs{\norm{M(\om _1,\om _2,\om _3)}{\ell ^{2/\th}_{-\th (\frac{d}{2}-1)}}\hz\lec \prod _{l=1}^3\tnorm{\om _l}{\ell ^2_{s_e+\frac{\th}{3}}},\qquad \norm{M(\om _1,\om _2,\om _3)}{\ell ^{2/\th}_{-\th (\frac{d}{2}-1)}}\hz\lec \tnorm{\om _q}{\ell ^{2/\th}_{-\th (\frac{d}{2}-1)}}\prod _{\mat{l=1\\ l\neq q}}^3\tnorm{\om _l}{\ell ^2_{s_2}},\\
\norm{M_\mu (\om _1,\om _2,\om _3)}{\ell ^{2/\th}_{-\th (\frac{d}{2}-1)}}\lec \tnorm{\om _q}{\ell ^{2/\th}_{-\th (\frac{d}{2}-1)}}\prod _{\mat{l=1\\ l\neq q}}^3\tnorm{\om _l}{\ell ^2_{\frac{d}{2}-\frac{1+\th}{2}+}}}
for $\th \in [0,1]$.
To minimize the lower bound of regularity $\max \{ s_e+\frac{\th}{3},\frac{d}{2}-\frac{1+\th}{2}\}$, we choose $\th =\frac{2d-3}{5}$.
From the resulting estimates, we obtain (B3), (B2)', and (B2), respectively.
\end{proof}


\bigskip
\bigskip


\bigskip
\begin{thebibliography}{99}
\bibitem{B93-1} J. Bourgain, \textit{Fourier transform restriction phenomena for certain lattice subsets and applications to nonlinear evolution equations, I, Schr\"odinger equations}, Geom. Funct. Anal. \textbf{3} (1993), 107--156.
\bibitem{B07} J. Bourgain, \textit{On Strichartz's inequalities and the nonlinear Schr\"odinger equation on irrational tori}, Mathematical aspects of nonlinear dispersive equations, 1--20, Ann. of Math. Stud., \textbf{163}, Princeton Univ. Press, Princeton, NJ, 2007.
\bibitem{B13} J. Bourgain, \textit{Moment inequalities for trigonometric polynomials with spectrum in curved hypersurfaces}, Israel J. Math. \textbf{193} (2013), no. 1, 441--458.
\bibitem{BD15} J. Bourgain and C. Demeter, \textit{The proof of the $l^2$ decoupling conjecture}, Ann. of Math. (2) \textbf{182} (2015), no. 1, 351--389.
\bibitem{CW10} F. Catoire and W.-M. Wang, \textit{Bounds on Sobolev norms for the defocusing nonlinear Schr\"odinger equation on general flat tori}, Commun. Pure Appl. Anal. \textbf{9} (2010), no. 2, 483--491.
\bibitem{CH19} X. Chen and J. Holmer, \textit{The derivation of the $\mathbb{T}^3$ energy-critical NLS from quantum many-body dynamics}, Invent. Math. \textbf{217} (2019), no. 2, 433--547.
\bibitem{CCT03p} M. Christ, J. Colliander, and T. Tao, \textit{Instability of the periodic nonlinear Schr\"odinger equation}, preprint (2003). Available at \texttt{arXiv:math/0311227}
\bibitem{CKSTT04} J. Colliander, M. Keel, G. Staffilani, H. Takaoka, and T. Tao, \textit{Multilinear estimates for periodic KdV equations, and applications}, J. Funct. Anal. \textbf{211} (2004), no. 1, 173--218.
\bibitem{D14p} C. Demeter, \textit{Incidence theory and restriction estimates}, preprint (2014). Available at \texttt{arXiv:1401.1873}
\bibitem{D17} S. Demirbas, \textit{Local well-posedness for 2-D Schr\"odinger equation on irrational tori and bounds on Sobolev norms}, Commun. Pure Appl. Anal. \textbf{16} (2017), no. 5, 1517--1530.
\bibitem{GKO13} Z. Guo, S. Kwon, and T. Oh, \emph{Poincar\'e-Dulac normal form reduction for unconditional well-posedness of the periodic cubic NLS}, Comm. Math. Phys. \textbf{322} (2013), no. 1, 19--48.
\bibitem{GO18} Z. Guo and T. Oh, \textit{Non-existence of solutions for the periodic cubic NLS below $L^2$}, Int. Math. Res. Not. IMRN \textbf{2018}, no. 6, 1656--1729.
\bibitem{GOW14} Z. Guo, T. Oh, and Y. Wang, \textit{Strichartz estimates for Schr\"odinger equations on irrational tori}, Proc. Lond. Math. Soc. (3) \textbf{109} (2014), no. 4, 975--1013.
\bibitem{HS19} S. Herr and V. Sohinger, \textit{Unconditional uniqueness results for the nonlinear Schr\"odinger equation}, Commun. Contemp. Math. \textbf{21} (2019), no. 7, 1850058, 33 pp.
\bibitem{HTT11} S. Herr, D. Tataru, and N. Tzvetkov, \textit{Global well-posedness of the energy-critical nonlinear Schr\"odinger equation with small initial data in $H^1(\T ^3)$}, Duke Math. J. \textbf{159} (2011), no. 2, 329--349.
\bibitem{HTT12} S. Herr, D. Tataru, and N. Tzvetkov, \textit{Strichartz estimates for partially periodic
solutions to Schr\"odinger equations in $4d$ and applications}, J. reine angew. Math., in press. DOI:10.1515/crelle-2012-0013
\bibitem{J26} V. Jarn\'ik, \textit{Uber die Gitterpunkte auf konvexen Kurven} (German), Math. Z. \textbf{24} (1926), 500--518.
\bibitem{K95} T. Kato, \textit{On nonlinear Schr\"odinger equations. II. $H^s$-solutions and unconditional well-posedness}, J. Anal. Math. \textbf{67} (1995), 281--306.
\bibitem{KV16} R. Killip and M. Vi\c{s}an, \textit{Scale invariant Strichartz estimates on tori and applications}, Math. Res. Lett. \textbf{23} (2016), no. 2, 445--472.
\bibitem{K14} N. Kishimoto, \textit{Remark on the periodic mass critical nonlinear Schr\"odinger equation}, Proc. Amer. Math. Soc. \textbf{142} (2014), no. 8, 2649--2660.
\bibitem{K-all} N. Kishimoto, \textit{Unconditional uniqueness of solutions for nonlinear dispersive equations}, preprint (2019). Available at \texttt{arXiv:1911.04349}
\bibitem{L19} G.E. Lee, \textit{Local wellposedness for the critical nonlinear Schr\"odinger equation on $\mathbb{T}^3$}, Discrete Contin. Dyn. Syst. \textbf{39} (2019), no. 5, 2763--2783.
\bibitem{M09} L. Molinet, \textit{On ill-posedness for the one-dimensional periodic cubic Schr\"odinger equation}, Math. Res. Lett. \textbf{16} (2009), no. 1, 111--120.
\bibitem{S14} N. Strunk, \textit{Strichartz estimates for Schr\"odinger equations on irrational tori in two and three dimensions}, J. Evol. Equ. \textbf{14} (2014), no. 4-5, 829--839.
\bibitem{W13} Y. Wang, \textit{Periodic nonlinear Schr\"odinger equation in critical $H^s(\T ^n)$ spaces}, SIAM J. Math. Anal. \textbf{45} (2013), no. 3, 1691--1703. 
\end{thebibliography}
\end{document}